\documentclass[11pt,a4paper]{amsart} 
\usepackage{amssymb}
\usepackage[dvips]{graphicx}
\usepackage[latin1]{inputenc}
\usepackage[left=1in, right=1in, bottom=2in, top=1in]{geometry}

\numberwithin{equation}{section}
\newtheorem{thm}{Theorem}
\newtheorem{prop}[thm]{Proposition}
\newtheorem{lemma}[thm]{Lemma}
\newtheorem{cor}[thm]{Corollary}

\newtheorem{definition}[thm]{Definition}

\theoremstyle{definition}

\newcommand{\N}{\mathbb{N}}

\newcommand{\R}{\mathbb{R}}

\newcommand{\Z}{\mathbb{Z}}

\numberwithin{thm}{section}

\title{Symmetric semi-algebraic sets and non-negativity of symmetric polynomials}
\author{Cordian  Riener}
\address{Aalto Science Institute\\
PO Box  11000\\
FI-00076 Aalto, Finland  }

\begin{document}

\begin{abstract}
The question of how to certify the non-negativity of a polynomial function lies at the heart of Real Algebra and  has important applications to optimization. 
Timofte\cite{timofte-2003} provided a useful way of certifying non-negativity of symmetric polynomials. 
In this note we  slightly generalize  Timofte's statement and investigate families of polynomials that allow special representations in terms of power-sum polynomials.
We also recover the consequences of Timofte's original statements as a corollary.
\end{abstract}
\maketitle
\section{Introduction}
Real Algebraic Geometry evolved around the question how to certify  that a polynomial function $f\in\R[X_1,\ldots,X_n]$ assumes only nonnegative values and  the study of so called semi-algebraic sets. These  subsets of $\R^n$ are defined by  a Boolean formula whose atoms are polynomial equalities and inequalities. Given  polynomials $f_1\ldots,f_m\in\R[X_1,\ldots,X_n]$ we will denote by $S(f_1,\ldots,f_m)\subset\R^n$ a semi-algebraic set, such that the equalities and inequalities appearing in the description  are given by  the polynomials $f_1,\ldots,f_m$. With this setup, problems such as deciding if  a semi-algebraic set is empty or not, or computing topological invariants of such sets are central to Real Algebraic Geometry \cite{BPR}.  These algorithmic questions also have applications to other areas of mathematics, for example to optimization. In recent years there has been some interest to this question in particularly structured situations, for example polynomials  invariant under the action of a group. In this note  we investigate  symmetric polynomials, i.e., polynomials invariant under all permutations of the variables. In this setting Timofte \cite{timofte-2003} introduced  the following so-called \emph{half-degree principle}: Let $f$ be a  symmetric polynomial in $n$ real variables of degree $2d$. Then, if the inequality $f(y)\geq 0$ holds on all points $y\in\R^n$ that do not have more than $d$ distinct components, it is valid for all $y\in\R^n$. Furthermore, the called the \emph{degree principle} applies for a general  semi-algebraic set $S$  described  by symmetric polynomials $f_1,\ldots,f_m$ such that every symmetric polynomial $f_i$ has at most degree $d$. In this case emptiness of $S$ can be certified by restricting to the points with at most $d$ distinct components. 
Both the half degree principle and the degree principle allow for a more efficient way to     check for non-negativity of a symmetric polynomial or to certify emptiness of a symmetric semi-algebraic set. For $k\in\Z$ a $k$-partition of $n$ (denoted as $\vartheta\vdash_k n$) is an ordered sequence $\vartheta:=(\vartheta_1,\ldots,\vartheta_k)$ of non-negative integers with $\vartheta_1\geq \vartheta_2\geq\ldots,\geq \vartheta_k$ and $\vartheta_1+ \vartheta_2+\ldots+ \vartheta_k=n$.  Given a symmetric polynomial  $f$ and  $\vartheta\vdash_k n$  one can define a $k$-variate polynomial 
$f^{\vartheta}\in\R[T_1,\ldots, T_k]$ via
$$f^{\vartheta}(T_1,\ldots,T_k):=f(\underbrace{T_1,\ldots,T_1}_{\vartheta_1},\underbrace{T_2,\ldots,T_2}_{\vartheta_2},\ldots,\underbrace{T_{k},\ldots,T_{k}}_{\vartheta_{k}}).$$

Using this notation the half-degree principle states  that if $f^{\vartheta}$  is non negative on $\R^k$  for all $\vartheta\vdash_k n$, then the original symmetric polynomial $f$ is non-negative. Since $\vartheta\vdash_k n$ implies that $\vartheta\in\{1,\ldots,n\}^k$,  the number of possible $k$ partitions of $n$ is  bounded by $n^k$. Hence the complexity of deciding if  a   symmetric polynomial of a fixed degree is non-negative depends only polynomially on $n$. This construction can be applied appropriately  to any semi-algebraic set $S$ defined by symmetric polynomials whose degrees are bounded by a fixed  number. This idea has  remarkable consequences, for example in  SDP- relaxations for optimization tasks defined by symmetric polynomials \cite{GC,RLJT},  or the study of topological complexity of projections of general (i.e. non-symmetric) semi-algebraic sets \cite{BS}. 
In the remainder of this paper we want to show that these remarkable statements and a slight generalization can be derived in an elementary way by properties of the power sum polynomials.  

\section{Generalizing the degree principle}

One of the main observations in the proof of the degree principle is the form of the representation of a symmetric polynomial of degree $d$ in terms of generators of the polynomial algebra of symmetric polynomials:
For integers $n,i\in\N$ define the power sum polynomials $$p_i^{(n)}:=\sum_{j=1}^n X_j^i.$$ We  will omit the superscript $n$ whenever the number of variables is clear. 
The following statement is well-known and sometimes referred to as the fundamental theorem of symmetric polynomials.
\begin{thm}\label{thm:sym}
Let  $f\in\R[X_1,\ldots,X_n]$ be symmetric, then there is a unique polynomial $g\in\R[Z_1,\ldots,Z_n]$ such that \begin{equation}\label{eq:1}f(X_1,\ldots,X_n)=g(p_1(X_1,\ldots,X_n),\ldots,p_n(X_1,\ldots,X_n)).\end{equation}
\end{thm} 
Since the decomposition in 
 \eqref{eq:1} is unique, a closer inspection gives the following.
\begin{cor}\label{cor:dar}
Let $f\in\R[X_1,\ldots,X_n]$ be symmetric of degree $d\leq n$. Then setting $k=\lfloor\frac{d}{2}\rfloor$ and $d'=\min\{d,n\}$, we have \begin{eqnarray*}
f(X_1,\ldots,X_n)=g_0(p_1(X_1,\ldots,X_n),\ldots,p_k(X_1,\ldots,X_n))\\+\sum_{j=k+1}^{d'} g_{j-k}(p_1(X_1,\ldots,X_n),\ldots,p_k(X_1,\ldots,X_n))\cdot p_{j}(X_1,\ldots,X_n)
\end{eqnarray*}
\end{cor}
The following definition  generalizes this property to symmetric polynomials that  are representable using few power sums.
\begin{definition}\label{def:one}
Let $J\subset \N$ be a set of cardinality $d$. Then a symmetric polynomial  $f$ is called $J-$sparse, if it admits a representation in terms of the power sums associated to $J$, i.e.,
$$ f= g(p_{j_1},\ldots,p_{j_d}),$$
where $j_l\in J$ and $g\in \R[Z_1,\ldots,Z_d]$.
\end{definition}
Note that  every symmetric polynomial of degree $d$ is $\{1,\ldots, d\}$ - sparse. 
The following observation   provides a simple method to verify if a given polynomial is $J$ -sparse. We denote by $V\in\R[X_1,\ldots,X_n]^{n\times n}$  the Vandermonde matrix, i.e., $V=(X_i^{j-1})_{i,j=1}^n$. Further, let $V^{-1}\in\R(X_1,\ldots,X_N)^{n\times n}$ be its inverse. Notice that this inverse matrix can be calculated via 
$$V^{-1}_{i,j}:=\left(\frac{(-1)^{n+i-1}e_{n-i}(X_1,\ldots,\hat{X}_j,\ldots X_n)}{\prod_{l=1,l\neq j}^n(X_j-X_l)}\right),$$ where $e_j(X):=\sum_{J\subset \{1,\ldots,n\} |J|=j}\prod_{i\in J} X_i$ is the $j-$th elementary symmetric polynomial and $\hat{X_j}$ denotes leaving out this variable.

\begin{prop}
Let $f\in\R[X]$ be a symmetric polynomial denote by $\nabla f\in{\R[X_1,\ldots,X_n]}^n$ its gradient vector and define $h(X)\in\R[X_1,\ldots,X_n]^{n}$ via $h(X):=(M\cdot \nabla f)^t$. If a set  $J\subset\{1,\ldots,n\}$ satisfies $h(X)_j= 0$ for all $j\not\in J$,   then  $f$ is $J$-sparse.
\end{prop}
\begin{proof}
Let $g\in\R[Z_1,\ldots,Z_n]$ be the unique polynomial such that $f=g(p_1,\ldots,p_n)$. It follows by the chain rule that $$\frac{\partial f}{\partial X_i}= \sum_{l=1}^n \frac{\partial g}{\partial Z_l}(p_1(X),\ldots,p_n(X))\cdot \frac{\partial p_l}{\partial X_i}=\sum_{i=1}^{n} l X_i^{l-1}\frac{\partial g}{\partial Z_l}(p_1(X),\ldots,p_n(X)).$$
Now let $D\in\R^{n\times n}$ be the diagonal matrix with diagonal entries $(1,\ldots,n)$. Then this identity can be written as
$$\nabla f= V\cdot D\cdot \nabla g,$$ and hence $\nabla g=D\cdot V^{-1}\nabla f$. Therefore, $g$ does not depend on $p_j$ if and only if $(\nabla g)_j=0$ which in turn is the case if and only if $(V^{-1}\nabla f)_j=0$, proving the statement.
\end{proof}

\begin{definition}
For $x:=(x_1,\ldots,x_n)\in\R^n$, let $\pi(x):=|\{x_1,\ldots,x_n\}|$ denote the number of distinct coordinates of $x$ and 
$\pi^{+}(x):=|\{x_1,\ldots,x_n\}\cap\R_{> 0}|$ denote the number of distinct positive coordinates of $x$. Further, for $k\in \N$ define the sets
$$A_{k}:=\{ x\in\R^n \,:\, \pi(x)\leq k\} \text{ and } A_{k}^+:=\{ x\in\R^n_{\geq 0} \,:\, \pi^{+}(x)\leq k\}.$$
\end{definition}
The following theorem is a generalization of the degree principle.
\begin{thm}\label{thm:main}
Let   $J:=\{j_1,\ldots,j_d\}\subset\N$ with $j_1< j_2\ldots< j_d$  and let $f_1,\ldots,f_m\in \R[X]$ be symmetric   $J-$sparse polynomials. 
Consider a non-empty semi-algebraic set $S(f_1,\ldots,f_m)\subset\R^n$. Then the following holds:
\begin{enumerate}
\item If all of $j_1,\ldots,j_d$ are even, then $S$ contains a non-negative point with at most $d$ distinct non-zero coordinates, i.e., $S\cap A_d^+\neq \emptyset$.
\item If at least one of $j_1,\ldots,j_d$ is odd, then $S$ contains a point with at most $\ell:=\min\{j_d,2d+1\}$ many distinct coordinates, i.e., $S\cap A_\ell\neq \emptyset$.
\end{enumerate}
\end{thm}
In particular  we immediately recover the degree principle:
\begin{cor}[Degree Principle]
Let $S\subset \R^n$ be a semi-algebraic set defined  by symmetric  polynomials of degree at most $d$, then 
 $S\neq\emptyset$ if and only if $S\cap A_{k}\neq\emptyset$, where $k:=\max\{2,d\}$.
 \end{cor}
 The proof will use Descartes' rule of signs, a statement on the number of positive and negative roots of a univariate polynomial.
Let $$p(T)=\sum_{i=0}^n a_i T^i$$ be a univariate polynomial with real coefficients. 
We define $\nu$ to be the number of variations in sign of the sequence of coefficients $a_0,\ldots,a_n$, i.e. the number of values of times such that the sign of the sequence $a_0,\ldots,a_n$ changes.
\begin{prop}[Descartes' rule of signs]
Let $f(T)\in\R[T]$ be a univariate polynomial. Then the number of positive roots, i.e. $t\in(0,\infty)$ with $f(t)=0$ is at most $\nu$.
\end{prop}

\begin{proof}[Proof of Theorem \ref{thm:main}]
$(1)$ Assume that $j_1,\ldots,j_d$ are even. Fix $y\in S(f_1,\ldots,f_l)$, set $$a_1:=p_{j_1}(y),\ldots, a_d:=p_{j_d}(y),$$
and consider $$H(a_1,\ldots,a_d):=\{v\in\R^n\,:\,p_{j_1}(v)=a_1,\ldots,p_{j_d}(v)=a_d\}.$$ Clearly, $H(a_1,\ldots,a_d)\subset S(f_1,\ldots,f_l)$. Since all of $j_1,\ldots,j_d$ are even,   $H(a_1,\ldots,a_d)$ is compact.   Therefore, the polynomial function $p_{j_{d+2}}$ has an extreme point $\zeta\in H(a_1,\ldots,a_d)$, and since $p_{j_{d+2}}$ is an even function, we can assume that all  coordinates $\zeta_i$ of $\zeta$ are non-negative. Since $\zeta$ is an extreme point, it follows from  Lagrange's theorem that there are real $\lambda_0,\ldots,\lambda_d$, which are not all zero, such that
$$\lambda_0\nabla p_{j_{d+2}}(\zeta)=\sum_{l=1}^d \lambda_l \nabla p_{j_l}(\zeta).$$ 
Since $\frac{\partial p_i}{\partial X_j}=i\,X_j^{i-1}$, this in turn implies that each of the $\zeta_i$  is a non-negative root of the univariate polynomial 
$$\Lambda(T):=(j_{d}+2)T^{ j_{d+1}}-\sum_{l=1}^d j_l\alpha_l T^{j_l-1}.$$   This polynomial has at most $d+1$ non zero coefficients and hence by Descartes' rule  there are at most $d$ distinct positive roots. Therefore, $\zeta$ has at most $d$ distinct non-zero coordinates.

$(2)$ We start in the case, when at least one of  $j_1,\ldots ,j_d$ is even. Consider the set $H(a_1,\ldots,a_d)$, which is again compact and thus the function $p_{j_d+1}$ will have an extreme point.  With the same arguments the coordinates of the extreme points are roots of a  univariate polynomial $\Lambda'$ of degree $j_d$.    Descartes' Rule implies that $\Lambda'(T)$ can have at most $d$ distinct positive roots and at most $d$ distinct negative roots. So in total the maximal number of distinct real roots (including zero) is $\min\{j_{d},2d+1\}$.
Finally, suppose that all of $j_1,\ldots ,j_d$ are odd. Then $H(a_1,\ldots,a_d)$ can be unbounded. However, the function $p_2$ will have a minimum over $H(a_1,\ldots,a_d)$, and one can argue in the same manner.
\end{proof}
\section{Half-degree principle}
The half-degree principle, which applies in the case  of one polynomial (in)equality, an  iteven a stronger result can
be achieved. We give here an elementary proof for the half-degree principle. This idea of proof can easily be generalized to various situations where a set
is described by one polynomial (in)equality that has special representation in terms of
power sums.

\begin{thm}[Half-degree principle]\label{thm:hdp}
Let $f$ be a symmetric polynomial of degree $d$. Then
\begin{enumerate}
\item $f(y)\geq 0$ for all $y\in\R^n$ if and only if $f(\tilde{y})\geq 0$ for all $\tilde{y}\in A_{k}\neq\emptyset$, where $k=\max\{2,\lfloor \frac{d}{2} \rfloor\}$, 
\item $f(y)\geq 0$ for all $y\in\R_{\geq 0}^n$ if and only if $f(\tilde{y})\geq 0$ for all $\tilde{y}\in A_{k,k}\neq\emptyset$, where $k=\max\{\lfloor \frac{d}{2} \rfloor\}$
\item $V:=\{ y\in\R^n\,:\, f(y)=0\}\neq\emptyset$ if and only if $V\cap A_{k}\neq\emptyset$, where $k=\max\{2,\lfloor d \rfloor\}$.
\end{enumerate}
\end{thm}
Before we prove the Theorem we will state some technical properties in the following lemma. 
\begin{lemma}\label{le:tec}
Let $n>1$ and $\xi\in\R^n$ with $\pi(\xi)=n$. Then there exists $\delta>0$ such that for every $0<\varepsilon<\delta$ the ball $B_\varepsilon(\xi)$ around $\xi$ with radius $\varepsilon$ contains points $\zeta$ and $\nu$ with the following properties:
\begin{enumerate}
\item $p_{i}(\xi)=p_i(\zeta)=p_i(\nu)$ for all $i\in\{1,\ldots,n-1\}$,
\item $p_{n}(\zeta)<p_n(\xi)<p_n(\nu)$,
\item $\xi_i\not\in \{\zeta_1,\ldots, \zeta_n\}$ and $\xi_i\not\in \{\nu_1,\ldots, \nu_n\}$ for all $i\in\{1,\ldots,n\}$.
\end{enumerate}
\end{lemma}
\begin{proof}
For $(1)$ and $(2)$ it suffices  that the Jacobian of the map $\Pi:\R^n\rightarrow \R^n$, with $x\mapsto (p_1(x),\ldots,p_n(x))$ has full rank $n$ at $\xi$. The properties  $(1)$ and $(2)$  then are  consequences of the inverse function theorem.
For $(3)$ we consider the elementary symmetric functions $e_j(X)$ and define $b_j:=e_j(\xi)$. Then the coordinates of $\xi$ are exactly the $n$ roots of the polynomial $$f_\xi(T):=T^n+\,\sum_{i=1}^n (-1)^{i}b_i\,T^{n-1}.$$  Similarly, we  define polynomials $f_\zeta,f_\nu\in\R[T]$.  Since $e_1,\ldots, e_{n-1}$ are polynomials in $p_1,\ldots, p_{n-1}$ we have by  $(1)$ that these three univariate polynomials only differ at by a constant scalar. Thus it follows that 
$f_{\zeta}(\xi_i)=f_{\xi}(\xi_i)-(f_{\xi}(0)-f_\zeta(0))\neq 0$ and $f_{\nu}(\xi_i)=f_{\xi}(\xi_i)-(f_{\xi}(0)-f_\nu(0))\neq 0$ for all $i\in\{1,\ldots,n\}$,  proving $(3)$.
\end{proof}

\begin{proof}[Proof of Theorem \ref{thm:hdp}]

$(1):$  Let $k:=\max\{2,\lfloor \frac{d}{2} \rfloor\}$. We will show that for  $y\in\R^n$ there is $\tilde{y}\in A_{k}$ such that $f(\tilde{y})\leq f(y)$. 
Let $y\in\R^n$, define $a_1:=p_1(y),\ldots, a_k:=p_k(y)$, and consider $$H(a_1,\ldots,a_k):=\{\xi\in \R^n\,:\, p_1(\xi)=a_1,\ldots p_k(\xi)=a_k\}.$$ 
We now examine the the optimization problem
$$\min_{\xi\in H(a_1,\ldots,a_k)}f(\xi).$$ Let $M_f(a_1,\ldots,a_k)$ denote the set of minimizers. Since $k\geq 2$ the set $H(a_1,\ldots,a_k)$ is compact and thus $M_f(a_1,\ldots,a_k)\neq\emptyset$. 
We will show that  $M_f(a_1,\ldots,a_k)\cap A_k\neq \emptyset$. First  remark that by Corollary \ref{cor:dar} we have
$$f=g_0(p_1,\ldots,p_k)+\sum_{j=k+1}^{d'} g_{j-k}(p_1,\ldots,p_k)\cdot p_{j},$$ where $d'=\min\{d,n\}$. 
Therefore,  setting $\ell_j:=g_{j-k}(p_1(y),\ldots,p_k(y))$ we have that over $H(a_1,\ldots,a_k)$ the function $f$ is equivalent to a linear combination of power sums defined as $$\phi:=\sum_{i=k+1}^{d'}\ell_i p_i.$$
For every $\tilde{y}\in M(a_1,\ldots,a_k)$ the first order Lagrange condition has to hold, i.e., there are $\lambda_0,\ldots,\lambda_k\in\R$ which are not all zero, such that  every coordinate $\tilde{y}_j$ of $\tilde{y}$ is a root of  the equation 
\begin{equation}\label{eq:Lagrange}\lambda_0\left(\sum_{i=k+1}^{d'}i\,\ell_i \tilde{y}_j^{i-1} \right)\,+\,\sum_{i=1}^{k} i\,\lambda_i\tilde{y}_j^{i-1}  =0,\end{equation}
and it follows that $\pi(\tilde{y})\leq d'-1$. 
In the case that $\lambda_0\neq 0$ but $\lambda_i=0$ for all $i\neq 0$ this is clearly only possible if $\tilde{y}\in A_k$ or all $\ell_i=0$. Since in the second case this in turn implies that $f$ is constant on $H(a_1,\ldots,a_k)$, i.e. $M_f(a_1,\ldots, a_k)= H(a_1,\ldots,a_k)$ and further by Theorem \ref{thm:main} we have  $H(a_1,\ldots,a_k)\cap A_k\neq \emptyset$, we can conclude that $M_f(a_1,\ldots,a_k)\cap A_k\neq \emptyset$. 

In the case that at least one $\lambda_i\neq 0$ for $i\in\{1,\ldots,k\}$  choose $\tilde{y}\in M_f(a_1,\ldots,a_k)$ in such away that  $m:=\pi(\tilde{y})<n$ is maximal with respect to all points in $M_f(a_1,\ldots,a_k)$. Suppose that $\tilde{y}\notin A_k$, i.e., that $m>k$ and without loss of generality  assume that the first $m$ coordinates of $\tilde{y}$ are pairwise distinct. Now, consider  $f$ locally only as a function of these first $m$ coordinates. This restriction is given by the polynomial  $\tilde{f}(X_1,\ldots,X_m):=f(X_1,\ldots,X_m,\tilde{y}_{m+1},\ldots,\tilde{y}_n)$. 
Denoting $S_m\subset S_n$ the subgroup permuting only the first $m$ coordinates, 
 clearly $\tilde{f}$ is $S_m$ invariant and since $2m>d$ there is a representation of the from
\begin{equation}\label{eq:tilde}\tilde{f}=\tilde{g}_0(p_1^{(m)},\ldots,p_{m-1}^{(m)})+\tilde{g}_1(p_1^{(m)},\ldots,p_{m-1}^{(m)})\,p_m^{(m)}.\end{equation}
Lemma \ref{le:tec} now yields that in in every small enough (m-dimensional) ball around $(\tilde{y}_1,\ldots,\tilde{y}_m)$  there exist  points $\zeta\in\R^m$ and $\nu\in\R^m$  for which all but the last power sum agree with the evaluation on $(\tilde{y}_1,\ldots,\tilde{y}_m)$. Suppose that $\tilde{g}_1(p_1^{(m)}(\tilde{y}_1,\ldots,\tilde{y}_m),\ldots,p_{m-1}^{m}(\tilde{y}_1,\ldots,\tilde{y}_m))\neq 0$. Then it follows from the representation in \eqref{eq:tilde} it that $\tilde{f}(\zeta_1,\ldots,\zeta_m)<\tilde{f}(\tilde{y}_1,\ldots,\tilde{y}_m)$ or $\tilde{f}(\nu_1,\ldots,\nu_m)<\tilde{f}(\tilde{y}_1,\ldots,\tilde{y}_m)$, which clearly contradicts $\tilde{y}\in M(a_1,\dots,a_k)$.

Finally, suppose that $\tilde{g}_1(p_1^{(m)}(\tilde{y}_1,\ldots,\tilde{y}_m),\ldots,p_{m-1}^{m}(\tilde{y}_1,\ldots,\tilde{y}_m))= 0$. In this case it follows that $(\zeta_1,\ldots,\zeta_m,\tilde{y}_{m+1},\ldots,\tilde{y}_n) \in M(a_1,\ldots,a_k)$. However  by $(3)$ in Lemma \ref{le:tec} we can infer that $|\{\zeta_1,\ldots,\zeta_m,\tilde{y}_{m+1},\ldots,\tilde{y}_n)\}|>m$, which  contradicts the choice of $\tilde{y}$. Therefore we can conclude that $M(a_1\ldots,a_k)\subseteq A_k$.
  
$(2):$ Just observe that the polynomial $f$ is copositive if $\tilde{f}:=f(X_1^2,\ldots,X_n^2)$ non-negative. Then observing that $\tilde{f}$  has a representation with only even power sums, the arguments follow the same ideas.
 
$(3):$ Suppose $V\neq \emptyset$. Then $\min_{x\in\R^n} f(x)\leq 0$ and $\min_{x\in\R^n} -f(x)\geq 0$. Therefore it follows from $(1)$ that there is $y_1, y_2\in A_{k}$ such that $f(y_1)\leq 0$ and $f(y_2)\geq 0$. Since $A_k$ is connected the statement follows. \end{proof}

{\bf Open question:}

The proof presented in this article used properties of the power sum polynomials. In particular the notion of sparsity in Definition \ref{def:one} depends on this particular choice of generators for the ring of symmetric polynomials and  would be interesting to study if the results in Theorem  \ref{thm:main} are dependent on this choice.

{\it Acknowledgment:} 
The author thanks Andrew Arnold, Alexander Kova\v{c}ec, and an anonymous referee for many valuable comments on an earlier version that helped to improve the presentation of the results.

\end{document}